\documentclass[12pt]{amsart}
\usepackage[margin=1in]{geometry}
\usepackage{hyperref}

\usepackage{amsthm, amsfonts, amsmath, amssymb, amsxtra}
\numberwithin{equation}{section}
\usepackage{enumitem}

\theoremstyle{definition}
\newtheorem{theorem}{Theorem}[section]
\newtheorem{corollary}[theorem]{Corollary}
\newtheorem{lemma}[theorem]{Lemma}

\newtheorem*{corollary*}{Corollary} 
\newtheorem*{definition*}{Definition}

\newtheorem*{remark*}{Remark}
\newtheorem{remark}[theorem]{Remark}
\newtheorem{proposition}[theorem]{Proposition}
\newtheorem*{proposition*}{Proposition}

\newcommand{\norm}[1]{\ensuremath{\left\| {#1} \right\|}}

\DeclareMathOperator{\dist}{dist}
\newcommand{\convset}{\ensuremath{{\overline{S}[\varphi, \Omega_0]}}}
\newcommand{\omout}{\ensuremath{{\Omega\setminus\Omega_0}}}

\newcommand{\pp}{\ensuremath{\mathbf{p}}}
\newcommand{\ee}{\ensuremath{\varepsilon}}

\usepackage{cite}

\title[Approximation of Convex Funcionals with general Lagrangians]
{On Approximation of Convex Functionals with a Convexity constraint 
and general Lagrangians}
\author{Young Ho Kim}
\address{
	Department of Mathematics, Texas A\&M University, College Station, TX, 77843, USA \\
	Department of Mathematics, Indiana University, Bloomington, IN 47405, USA}
\email{yhkim@tamu.edu}

\begin{document} 


\begin{abstract}
	In this note, 
	we prove that minimizers of convex functionals with a convexity constraint 
	and a general class of Lagrangians can be approximated 
	by solutions to fourth-order equations of Abreu type.
	Our result generalizes that of 
	Le (Twisted Harnack inequality and approximation of variational problems with a convexity constraint by singular Abreu equations. \emph{Adv. Math.} \textbf{434} (2023))
	where the case of quadratically growing Lagrangians was treated.
\end{abstract}
\maketitle

\section{Introduction and statement of the main result}
In this note, 
we prove that minimizers of convex functionals with a convexity constraint 
and a general class of Lagrangians can be approximated 
by solutions to fourth-order equations of Abreu type.
The problem of approximating minimizers to convex functionals
with a convexity constraint
by solutions of fourth-order Abreu type equations 
has been studied by several authors 
\cite{Kim1,CR,LeCPAM,LePRS,LZ,Twisted_Harnack}.
Previous results were proved either in two dimensions \cite{LZ}
or under a quadratic growth assumption on the Lagrangians 
\cite{Kim1,LeCPAM,LePRS,Twisted_Harnack}.
By replacing the quadratic term in the approximation scheme,
we extend these results to the case with general Lagrangians in dimensions $n\geq 2$.

\subsection{Variational problem with a convexity constraint}
	Let $\Omega$ and $\Omega_0$ be 
	bounded, smooth, convex domains in $\mathbb{R}^n$ ($n\geq 2$)
	with $\Omega_0\Subset \Omega$.
	Suppose $\varphi \in C^5(\overline{\Omega})$ is convex
	and $F=F(x,z,\pp):\mathbb{R}^n\times\mathbb{R}\times\mathbb{R}^n\rightarrow\mathbb{R}$
	is a smooth Lagrangian that is convex 
	in the variables $\pp\in\mathbb{R}^n$ and $z\in\mathbb{R}$.
	Consider the variational problem
	\begin{equation}\label{eqn:hdvar}
	\begin{aligned}
		\inf_{u\in\convset}  \int_{\Omega_0} F(x, u(x), Du(x)) \, dx
		\equiv\inf_{u\in\convset} J(u)
		\text{,}
	\end{aligned}
	\end{equation}
	over the competitors $u$ with a convexity constraint given by
	\begin{equation}\label{eqn:hdvarcon}
	\begin{aligned}
		\overline{S}[\varphi,\Omega_0]
		=\{ u:\Omega\rightarrow\mathbb{R} \text{ convex, }
		u=\varphi \text{ on } \Omega\setminus\Omega_0 \}
		\text{.}
	\end{aligned}
	\end{equation}
	An example of a variational problem of the form (\ref{eqn:hdvar})--(\ref{eqn:hdvarcon})
	is the Rochet-Chon\'e model 
	for the monopolist problem in economics \cite{RC}.
	In this model, the Lagrangian is given by 
	$F(x,z,\pp) = (|\pp|^q/q - x\cdot \pp + z)\gamma (x)$,
	for $q\in(1,\infty)$ and a nonnegative Lipschitz function $\gamma$.

	Note that
	$\overline{S}[\varphi,\Omega_0]$ contains all convex functions 
	in $\Omega_0$, 
	which have convex extensions 
	that agree with a given convex function $\varphi$ outside $\Omega_0$.
	In addition to a Dirichlet boundary condition 
	$u=\varphi$ on $\partial\Omega_0$,
	this constraint imposes, in some weak sense,
	a restriction on the gradient of $u$ at the boundary of $\Omega_0$.
	Consequently,
	it is hard to write a tractable Euler-Lagrange equation 
	for the variational problem (\ref{eqn:hdvar})--(\ref{eqn:hdvarcon}).
	Furthermore, 
	variational problems of this type are 
	difficult to handle in numerical schemes \cite{BCMO, M16}.
	Therefore, 
	one may ask whether minimizers of these problems can be approximated 
	by solutions of higher-order, well-posed equations.

	To address these difficulties,
	Carlier and Radice \cite{CR} introduced an approximation scheme 
	using solutions to Abreu equations in
	the case when the Lagrangian $F=F(x,z)$ does not depend
	on the gradient variable 
	$\pp=(p_1, \cdots, p_n)\in\mathbb{R}^n$.
	This was extended by Le \cite{LeCPAM} to the case when 
	$F$ can be split into 
	\begin{equation}\label{eqn:Fsplit}
	F(x,z,\pp)=F^0(x,z) + F^1(x,\pp) 
	\end{equation}
	with suitable conditions on $F^0$ and $F^1$.
	In these schemes,
	penalizations of the form 
	\begin{equation}\label{eqn:approx}
	\begin{aligned}
	J(v)
	+\frac{1}{2\ee}\int_{\Omega\setminus\Omega_0} (v-\varphi)^2\, dx
	-\ee \int_\Omega \log \det D^2 v\, dx
	\end{aligned}
	\end{equation}
	were introduced for small $\ee >0$.
	The idea behind the logarithmic penalization is that 
	it should act as a good barrier for the convexity constraint
	(\ref{eqn:hdvarcon}) in problems like (\ref{eqn:hdvar}).
	Looking at the critical point $u_\ee$ of (\ref{eqn:approx})
	where $F$ is given by (\ref{eqn:Fsplit}),
	we obtain an equation of the following form:
	\begin{align}\label{eqn:el}
		\left\{
		\begin{aligned}
		\ee U_\ee^{ij} D_{ij} w_\ee &= f_\ee 
			:= \left\{ \dfrac{\partial F^0}{\partial z} (x, u_\ee) 
		- \dfrac{\partial}{\partial x_i} 
		\left( \dfrac{\partial F^1}{\partial p_i} (x,Du_\ee) \right)
		\right\} \chi_{\Omega_0} +
		\dfrac{1}{\ee} (u_\ee-\varphi) \chi_{\Omega\setminus\Omega_0} 
		&&\text{ in } \Omega\text{,}\\
			w_\ee &= ({\det D^2 u_\ee})^{-1} 
				&&\text{ in }\Omega\text{.}\\
		\end{aligned}
		\right.
	\end{align}
	Here, $(U^{ij}_\ee)_{1\leq i,j\leq n}=(\det D^2u_\ee) (D^2u_\ee)^{-1}$
	is the cofactor matrix of the Hessian matrix $D^2u_\ee$
	and $\chi_E$ is the characteristic function of the set $E$.

	Note that (\ref{eqn:el}) is a system of two equations, 
	where one is a Monge-Amp\`ere equation for $u_\ee$:
	\begin{equation}\label{eqn:ma_expl} 
		\det D^2 u_\ee = w_\ee^{-1}
		\quad \text{in }\Omega\text{,}
	\end{equation}
	and the other is a linearized Monge-Amp\`ere equation for $w_\ee$:
	\begin{equation}\label{eqn:lma_expl}
		U_\ee^{ij} D_{ij} w_\ee = \ee^{-1} f_\ee 
		\quad\text{in }\Omega\text{.}
	\end{equation}
	Equation (\ref{eqn:lma_expl}) is called a linearized Monge-Amp\`ere equation 
	because $U_\ee^{ij} D_{ij}$ comes from linearizing the
	Monge-Amp\`ere operator $\det D^2 u_\ee$:
	\begin{align*} 
		\det D^2(u_\ee+tv) =  \det D^2 u_\ee + (U^{ij}_\ee D_{ij}v) t + 
		\cdots + (\det D^2 v)t^n
		\text{.}
	\end{align*}
	As $w_\ee = (\det D^2u_\ee)^{-1}$ is of second-order in $u_\ee$,
	(\ref{eqn:el}) is a fourth-order equation in $u_\ee$. 
	Because (\ref{eqn:el}) is a system of these two equations,
	it is natural to consider second boundary value problems for 
	(\ref{eqn:el}) with Dirichlet boundary conditions on $u_\ee$ and $w_\ee$,
	such as
	\begin{equation}\label{eqn:bdcon}
	  u_\ee=\varphi\text{,}\quad 
	  w_\ee=\psi\quad\text{on }\partial\Omega
	  \text{.}
	\end{equation}
	
	When $\ee^{-1} f_\ee$ in (\ref{eqn:ma_expl})--(\ref{eqn:lma_expl}) is replaced by $-1$,
	\begin{equation*} 
		U_\ee^{ij}D_{ij} [(\det D^2u_\ee)^{-1}] = -1
	\end{equation*}
	is the Abreu equation \cite{Ab}
	which appears in the problem of 
	finding K\"ahler metrics of constant scalar curvature 
	for toric manifolds \cite{D1,D2}.
	The term 
	\begin{equation*}
		\dfrac{\partial}{\partial x_i} 
		\left( \dfrac{\partial F^1}{\partial p_i} (x,Du_\ee) \right)
	\end{equation*} 
	in (\ref{eqn:el}) depends on $D^2u_\ee$,
	which is only guaranteed to be a matrix-valued measure
	under the assumption that $u_\ee$ is convex.
	Hence, (\ref{eqn:el}) is called a \emph{singular Abreu equation}
	\cite{KLWZ,LeCPAM,LePRS,LZ}.
	
	The general scheme is to first establish 
	the existence of solutions $(u_\ee)_{\ee>0}$ to 
	the second boundary value problem to
	Abreu-type equations of the form (\ref{eqn:el})
	with boundary conditions like (\ref{eqn:bdcon}),
	and then prove that 
	after passing to a subsequence $\ee_k\rightarrow 0$,
	solutions $(u_{\ee_k})_k$ converge uniformly on compact subsets of $\Omega$
	to a minimizer of the variational problem 
	(\ref{eqn:hdvar})--(\ref{eqn:hdvarcon}).
	Therefore, 
	solvability of second boundary value problem to 
	Abreu-type equations plays a critical role 
	in approximating minimizers to convex functionals with 
	a convexity constraint.
	For gradient-dependent Lagrangians,
	previous results were proved either in two dimensions \cite{LZ}
	or under a quadratic growth assumption on the Lagrangian 
	\cite{LeCPAM,LePRS,Twisted_Harnack,Kim1}.
	By replacing the quadratic term in the approximation scheme,
	we extend the results to the case with  general Lagrangians
	in dimensions $n\geq 2$;
	also see Remark \ref{rmk:end}.

	\subsection{The Main Result}
	In this note,
	we prove that the approximation scheme 
	for the variational problem (\ref{eqn:hdvar})--(\ref{eqn:hdvarcon})
	using Abreu-type equations can be 
	extended to a general class of  Lagrangians $F$ that 
	do not necessarily satisfy a quadratic growth assumption in dimensions $n\geq 2$. 
	We achieve this by modifying the approximation scheme in (\ref{eqn:approx}). 

	Instead of a quadratic growth condition on the Lagrangian $F=F(x,z,\pp)$ 
	in the $\pp$ variable, 
	we assume that $F$ satisfies the following conditions:
	\begin{enumerate} 
		\item[(F1)] $F$ is smooth, and convex in variables
			$z\in\mathbb{R}$ and $\pp\in\mathbb{R}^n$.
		\item[(F2)] 
		The derivatives of $F$ satisfy the following growth estimates 
			for $z\in\mathbb{R}$, $\pp\in\mathbb{R}^n$:
			\begin{equation}\label{Fest} 
			\begin{aligned}
				\left| \frac{\partial F}{\partial z} (x,z,\pp) \right| 
				+ \left| \frac{\partial F}{\partial p_i} (x,z,\pp) \right| 
				&\leq f_0 (|z|) g_0 (|\pp|)
				\quad\text{for all }1\leq i\leq n\text{,}\\
				0 \leq (F_{p_i p_j} (x,z,\pp))_{1\leq i,j\leq n} 
				&\leq f_1 (|z|) g_1 (|\pp|) I_n\text{,} \\
				|F_{p_i x_i} (x,z,\pp)| &\leq f_2(|z|) g_2(|\pp|)\text{,}\\
				|F_{p_i z} (x,z,\pp)| &\leq f_3(|z|) g_3(|\pp|)
				\quad\text{for all }1\leq i\leq n\text{.}
			\end{aligned}
			\end{equation}
			Here $f_k$, $g_k$ ($0\leq k\leq 3$) are smooth, convex and 
			increasing functions from $[0,\infty)$ to $[0,\infty)$,
			$I_n$ is the $n\times n$ identity matrix,
			and repeated indices are summed.
	\end{enumerate}
	The convexity assumptions on $f_k$, $g_k$ are reasonable 
	as any smooth, increasing growth function 
	$\eta:[0,\infty)\rightarrow[0,\infty)$ 
	can be replaced by 
	\begin{align*} 
		\widetilde{\eta}(x):=\int_0^{x+1} \eta(s) \, ds
	\end{align*}
	which is convex, smooth, increasing and satisfies
	$\widetilde{\eta}\geq\eta$.

	Now, we will introduce the modifications made to the approximating functional 
	(\ref{eqn:approx}). 
	The first modification comes from Le \cite{LePRS,Twisted_Harnack}.
	Let $\rho$ be a uniformly convex defining function of $\Omega$, that is,
	\begin{align*} 
		\{x\in\mathbb{R}^n \mid \rho(x)<0\} = \Omega
		\text{,}\quad
		\rho = 0\quad\text{on }\partial\Omega
		\text{,}\quad\text{and }D\rho \neq 0
		\quad\text{on }\partial\Omega\text{.}
	\end{align*}
	Now, for $\ee >0$, we set 
	\begin{align} 
		\widetilde{\varphi}_\ee(x) = 
		\varphi(x) + \ee^{\frac{1}{3n^2}} ( e^{\rho(x)} - 1)
		\text{.}
	\end{align}
	In the quadratic term 
	$$\frac{1}{2\ee}\int_{\Omega\setminus\Omega_0} (u-\varphi)^2\, dx$$ 
	from (\ref{eqn:approx}), 
	we replace $\varphi$ by $\widetilde{\varphi}_\ee$.
	This makes the new function `sufficiently' uniformly convex
	and makes it possible to handle Lagrangians $F$ that are 
	non-uniformly convex.  

	Furthermore, we replace the quadratic term 
	$$\frac{1}{2\ee}\int_{\Omega\setminus\Omega_0} (u-\widetilde{\varphi}_\ee)^2\, dx$$
	again by  
	$$\frac{1}{\ee}\int_{\Omega\setminus\Omega_0} G(u-\widetilde{\varphi}_\ee)\, dx\text{,}$$
	where $G$ is a suitable function to be defined later.
	In Le \cite{LeCPAM,LePRS,Twisted_Harnack},
	a quadratic growth assumption had to be imposed on $F$
	as the integral including the derivative $F^1_{p_i x_i}$
	had to be bounded by the quadratic term in the approximation scheme; 
	see \cite[inequality (4.11)]{LZ},
	\cite[inequalities (2.4) and (4.15)]{LeCPAM},
	and \cite[inequalities (1.9) and (3.12)]{LePRS}.
	In this note, 
	this modification makes it possible to 
	remove the quadratic growth assumption on $F$;
	also see Remark \ref{rmk:G}.

	Because $f_k$, $g_k$ are smooth, convex, increasing and nonnegative,
	if we define 
	\begin{equation}\label{eqn:hdef}
	\begin{aligned}
		H(x) =x(1 + f_0(x) g_0(x) + f_2(x) g_2 (x) + xf_3(x) g_3(x))
		\text{,}
	\end{aligned}
	\end{equation}
	then $H$ is a convex, smooth, and increasing function 
	from $[0, \infty)$ to $[0,\infty)$ with $H(x)\geq x$. 
	Now, we define the convex function $G$ by
	\begin{align}\label{eqn:Hdef}
		G(x) = \int_0^{x^2} H(t) \, dt
		\text{.}
	\end{align}

	With these modifications to (\ref{eqn:approx}),
	the approximating functional used in this note will be 
	\begin{align}\label{Jepsilon}
		J_\ee (u) = \int_{\Omega_0} F(x,u(x), Du(x)) \, dx
		+ \frac{1}{\ee} \int_\omout G(u - \widetilde{\varphi}_\ee) \, dx 
		- \ee \int_\Omega \log \det D^2 u(x) \, dx
		\text{,}
	\end{align}
	and our second boundary problem becomes
	\begin{equation}\label{eqn:eleq} 
	\begin{aligned}
	\left\{
		\begin{aligned} 
			\ee U_\ee^{ij} D_{ij} w_\ee & = f_\ee  \\
													   & := 
				\left(
					\dfrac{\partial F}{\partial z} (x, u_\ee, Du_\ee) 
				- \dfrac{\partial}{\partial x_i} \left( \dfrac{\partial F}{\partial p_i} (x,u_\ee, Du_\ee) \right)
				\right)\chi_{\Omega_0}
				+
				\dfrac{G'(u_\ee - \widetilde{\varphi}_\ee)}{\ee} 				\chi_{\Omega\setminus\Omega_0}
			&&\text{in }\Omega\text{,} \\
			w_\ee &= ({\det D^2 u_\ee})^{-1} &&\text{ in } \Omega\text{,} \\
			u_\ee &= \varphi 
			\text{,}\quad
			w_\ee = \psi &&\text{ on } \partial\Omega\text{.}
		\end{aligned}
	\right.
	\end{aligned}
	\end{equation}
	Here 
	$(U_\ee^{ij})_{1\leq i,j\leq n}$ 
	is the cofactor matrix of $D^2 u_\ee$.

	Our main result is the following theorem.
	\begin{theorem}\label{mainthm} 
		Suppose $\Omega_0$ and $\Omega$ are  smooth
		and convex domains in $\mathbb{R}^n$ ($n\geq 2$), 
		where $\Omega$ is uniformly convex and $\Omega_0\Subset\Omega$. 
		Let $\varphi\in C^5(\overline{\Omega})$, $\psi\in C^3(\overline{\Omega})$,
		$\varphi$ is convex, and $\min_{\partial\Omega}\psi > 0$.  
		Let $F=F(x,z,\pp):\mathbb{R}^n\times\mathbb{R}\times\mathbb{R}^n\rightarrow\mathbb{R}$
		satisfy (F1)--(F2).  
		If $0<\ee<\ee_0<1$,
		where $\ee_0$ is a small number 
		depending only on $n$, $\Omega$, $\Omega_0$, $\varphi$, $\psi$, $f_k$, and $g_k$,
		then the following are true.
		\begin{enumerate}[label=(\roman*)] 
			\item The second boundary value problem (\ref{eqn:eleq}) with $G$ given by 
				(\ref{eqn:hdef})--(\ref{eqn:Hdef}) has a uniformly convex 
				$W^{4,s}(\Omega)$ solution $u_\ee$ for all $s\in (n,\infty)$.
				\label{thm:solvability} 
			\item Let $(u_\ee)_{0<\ee<\ee_0}$ be $W^{4,s}(\Omega)$ ($s>n$)
				solutions to (\ref{eqn:eleq}).
				Then, after passing to a subsequence $\ee_k\rightarrow 0$, 
				the sequence $(u_{\ee_k})_k$ 
				converges uniformly on compact subsets 
				of $\Omega$ to a minimizer $u$ of (\ref{eqn:hdvar})--(\ref{eqn:hdvarcon}).
				\label{thm:approx} 
		\end{enumerate}
	\end{theorem}

	\begin{remark} 
		In Le \cite{LeCPAM,LePRS,Twisted_Harnack},
		Lagrangians $F$ that satisfy a quadratic growth condition 
		in the $\pp$ variable are considered. 
		Compared to these results,
		Theorem \ref{mainthm} covers general Lagrangians in all dimensions $n\geq 2$
		that do not necessarily have a quadratic growth in the $\pp$ variable. 
		One example of such a Lagrangian would be given by 
		$F(x,z,\pp) = e^{|\pp|^2}$.
		This improvement comes from replacing the quadratic term in (\ref{eqn:approx}); 
		see Remark \ref{rmk:G}. 
	\end{remark}

	The rest of this note is organized as follows.
	In Section \ref{sec:solvability}, 
	we prove Theorem \ref{mainthm}\ref{thm:solvability}. 
	In Section \ref{sec:approx},
	we prove Theorem \ref{mainthm}\ref{thm:approx}.

\section{A priori estimates and existence of solutions}\label{sec:solvability}
In this section,
we prove Theorem \ref{mainthm}\ref{thm:solvability} using degree theory
and the a priori $W^{4,s}(\Omega)$ estimate in Proposition \ref{prop:unifest} below.
The proof mostly follows Le \cite[Section 2]{Twisted_Harnack}.
The main difference will be in proving the uniform $L^\infty$ bound for $u_\ee$
in Lemma \ref{prop:linftyest};
see Remark \ref{rmk:G}.

\begin{proposition}\label{prop:unifest}
	Suppose $u_\ee$ is a uniformly convex $W^{4,s}(\Omega)$ 
	($n<s<\infty$) solution to
	(\ref{eqn:eleq}),
	where $F$ satisfies (F1)--(F2) and $G$ is defined by
	(\ref{eqn:hdef})--(\ref{eqn:Hdef}).
	If $0<\ee<\ee_0<1$,
	where $\ee_0$ is a small number depending only on
	$n$, $\Omega$, $\Omega_0$, $\varphi$, $\psi$, $f_k$, and $g_k$,
	then there is $C(\ee)>0$ such that
	\begin{equation}\label{eqn:aprioriest}
	\begin{aligned}
		\norm{u_\ee}_{W^{4,s}(\Omega)}\leq C(\ee)
		\text{.}
	\end{aligned}
	\end{equation}
\end{proposition}

Fix $s\in(n,\infty)$.
Throughout the section, $u_\varepsilon$ will denote
a uniformly convex $W^{4,s}(\Omega)$ solution to (\ref{eqn:eleq}),
and we will use numbered constants $C_n$
to denote positive constants that do not depend on the solution $u_\varepsilon$
but only on $n$, $s$, $\Omega$, $\Omega_0$, $\varphi$, $\psi$, $f_k$, and $g_k$.
We will write $C_n$ for constants that do not depend on $\varepsilon$,
while for constants that depend on $\varepsilon$ the dependency will be explicitly stated.

We start by getting an $L^\infty$ bound for $u_\varepsilon$.

\begin{lemma}[Uniform $L^\infty$ bound on $u_\ee$]\label{prop:linftyest}
	If $0<\ee<\ee_0$
	where $\ee_0 = \ee_0(n,\Omega,\Omega_0,\varphi,\psi,f_k, g_k)$ is 
	a small number satisfying $\ee_0<1$,
	then
	\begin{equation}\label{eqn:linfbdd}
	\begin{aligned}
		||u_\ee||_{L^\infty(\Omega)} < C_{14}
		\text{.}
	\end{aligned}
	\end{equation}
\end{lemma}

	\begin{proof} 
		Consider $\ee<1$.
		First, as $u_\ee$ is convex, we have an upper bound:
		\begin{align*} 
			u_\ee\leq
			\sup_{\partial\Omega} u_\ee = \sup_{\partial\Omega}\varphi =: C_0
			\quad\text{in }\Omega\text{.}
		\end{align*}
		For the lower bound, we consider two cases as in Le-Zhou \cite[pp.27--28]{LZ}.
		
		\emph{Case 1}. $u_\ee(x_0) > \widetilde{\varphi}_\ee(x_0) - 1$ for some $x_0\in \Omega_0$.
		We have 
		\begin{align*}
			u_\ee(x_0) > \inf_{\Omega} \widetilde{\varphi}_\ee-1 
			&\geq -\sup_{\Omega}|\widetilde{\varphi}_{\ee}|- 1 \\
			&\geq -(\sup_{\Omega}|\varphi| 
			+\sup_{\Omega} |e^\rho -1| + 1) =: -C_1
			\text{.}
		\end{align*}
		Let $x\in\Omega\setminus\{x_0\}$ be arbitrary
		and set  $y$ to be the intersection of the ray 
		$\overrightarrow{xx_0}$ and $\partial\Omega$.
		By the convexity of $u_\ee$, we have
		\begin{equation}\label{eqn:a}
		\begin{aligned}
			-C_1 \leq u_\ee(x_0)
			\leq \frac{|x_0-y|}{|x-y|} u_\ee (x) 
			+ \left(1-\frac{|x_0-y|}{|x-y|}\right) u_\ee(y)
			\text{.}
		\end{aligned}
		\end{equation}
		Because $x_0\in \Omega_0$ and $y\in\partial\Omega$, 
		we have
		\begin{align}\label{eqn:b}
			\frac{|x_0-y|}{|x-y|} \geq
			\frac{\dist(\Omega_0, \partial\Omega)}{\mathrm{diam}(\Omega)}>0
			\text{,}\quad\text{and }
			|u_\ee (y)|\leq \sup_{\partial\Omega} |\varphi|
			\text{.}
		\end{align}
		Combining (\ref{eqn:a}) and (\ref{eqn:b}) yields 
		a lower bound for $u_\ee$ in $\Omega$.

		\emph{Case 2}. $u_\ee \leq \widetilde{\varphi}_\ee - 1$ in $\Omega_0$.
		We will use the following inequality \cite[(3.6)]{LePRS}:
		\begin{align}\label{ineq1} 
			\int_{\partial\Omega} \ee ((u_\ee )_\nu^+ )^n \, dS 
			\leq C_2 + \int_\Omega -f_\ee (u_\ee - \widetilde{\varphi}_\ee) \, dx
			\text{.}
		\end{align}
		Here $\nu$ is the outer unit normal vector to $\partial\Omega$.
		Substituting $f_\ee$ from (\ref{eqn:eleq}) 
		and expanding the divergence term
		\begin{align}\label{fe_nondiv}
			\frac{\partial}{\partial x_i} \left( \frac{\partial F}{\partial p_i} 
				(x,u_\ee, Du_\ee) \right)
			= F_{p_i x_i} + F_{p_i z} D_i u_\ee + F_{p_i p_j} D_{ij} u_\ee
			\text{,}
		\end{align}
		we find that the integral in the right-hand side of (\ref{ineq1}) becomes
		\begin{equation}\label{eqn0}
		\begin{aligned}
			\int_\Omega -f_\ee (u_\ee - \widetilde{\varphi}_\ee ) \, dx
			&= -\frac{1}{\ee} \int_{\Omega\setminus\Omega_0} 
			G' (u_\ee - \widetilde{\varphi}_\ee) (u_\ee - \widetilde{\varphi}_\ee)\, dx \\
			&+\int_{\Omega_0} -F_z (u_\ee - \widetilde{\varphi}_\ee) \, dx 
			+\int_{\Omega_0} F_{p_i p_j} D_{ij} u_\ee (u_\ee - \widetilde{\varphi}_\ee) \, dx \\
			&+\int_{\Omega_0} F_{p_i x_i} (u_\ee - \widetilde{\varphi}_\ee) \, dx 
			+\int_{\Omega_0} F_{p_i z} D_i u_\ee (u_\ee-\widetilde{\varphi}_\ee) \, dx
			\text{.}
		\end{aligned}
		\end{equation}
		We estimate the terms in the right-hand side of (\ref{eqn0}) separately.

		First, as $(F_{p_i p_j})_{1\leq i,j\leq n}$ and $D^2 u_\ee$ are nonnegative-definite, 
		$F_{p_i p_j} D_{ij}u_\ee \geq 0$. 
		Since $u_\ee \leq \widetilde{\varphi}_\ee$,
		we get
		\begin{align}\label{est1}
			\int_{\Omega_0} F_{p_i p_j} D_{ij} u_\ee (u_\ee -\widetilde{\varphi}_\ee)\, dx \leq 0
			\text{.}
		\end{align}
		Next, we estimate $\int_{\Omega_0} -F_z(u_\ee-\widetilde{\varphi}_\ee)$.
		As $u_\ee$ is convex, we have the following gradient bound:
		\begin{align}\label{gradbound}
			|Du_\ee (x) | \leq \frac{\sup_{\partial\Omega} u_\ee - u_\ee (x)}
			{\dist (x,\partial\Omega)} \quad\text{for } x \in \Omega\text{.}
		\end{align}
		Therefore, for $x \in \Omega_0$, we have
		\begin{align}\label{eqn:gradestsubset}
			|Du_\ee (x) | \leq \frac{|\sup_{\partial\Omega}\varphi| + ||u_\ee||_{L^\infty(\Omega)}}{\dist(\Omega_0, \partial\Omega)}
			\leq C_3( 1 + ||u_\ee|| _{L^\infty(\Omega)} )
			\text{.}
		\end{align}
		Because $f_0$ and $g_0$ are increasing functions,
		(\ref{Fest}) and (\ref{eqn:gradestsubset}) give us 
		\begin{equation}\label{est2}
		\begin{aligned}
			\int_{\Omega_0} -F_z (u_\ee-\widetilde{\varphi}_\ee)
			&\leq \int_{\Omega_0} f_0\left(|u_\ee (x)|\right) 
			g_0\left(|Du_\ee(x)|\right) \left(|u_\ee(x)| + |\widetilde{\varphi}_\ee(x)|\right)\, dx \\
			&\leq |\Omega_0| f_0\left(||u_\ee||_{L^\infty(\Omega_0)}\right) 
			g_0\left(||Du_\ee||_{L^\infty(\Omega_0)}\right) 
			(||u_\ee||_{L^\infty(\Omega_0)} + ||\widetilde{\varphi}_\ee||_{L^\infty(\Omega_0)}) \\
			&\leq |\Omega_0| f_0\left(||u_\ee||_{L^\infty(\Omega_0)}\right) g_0(S) 
			(||u_\ee||_{L^\infty(\Omega_0)} + ||\widetilde{\varphi}_\ee||_{L^\infty(\Omega_0)}) \\
			&\leq Sf_0(S) g_0(S) 
			\text{,}
		\end{aligned}
		\end{equation}
		where 
		\begin{equation}\label{eqn:sdef}
		\begin{aligned}
			S = C_4(||u_\ee||_{L^\infty(\Omega_0)} + 1)
		\end{aligned}
		\end{equation}
		for some large $C_4>0$. 
		Other terms in the right-hand side of (\ref{eqn0})
		can be estimated similarly using (\ref{Fest}) and (\ref{eqn:gradestsubset}):
		\begin{equation}\label{est3}
		\begin{aligned}
			&\int_{\Omega_0} F_{p_i x_i} (u_\ee-\widetilde{\varphi}_\ee)\, dx \\
			&\leq n|\Omega_0| f_2\left(||u_\ee||_{L^\infty(\Omega_0)}\right) 
			g_2\left(||Du_\ee||_{L^\infty(\Omega_0)}\right) 
			(||u_\ee||_{L^\infty(\Omega_0)}+||\widetilde{\varphi}_\ee||_{L^\infty(\Omega_0)}) \\
			&\leq S f_2(S)g_2(S)
			\text{,}
		\end{aligned}
		\end{equation}
		and
		\begin{equation}\label{est4}
		\begin{aligned}
			&\int_{\Omega_0} F_{p_i z} D_i u_\ee (u_\ee-\widetilde{\varphi}_\ee)\, dx \\
			&\leq n|\Omega_0| f_3\left(||u_\ee||_{L^\infty(\Omega_0)}\right)
			g_3\left(||Du_\ee||_{L^\infty(\Omega_0)}\right) 
			\norm{Du_\ee}_{L^\infty(\Omega_0)}
			(||u_\ee||_{L^\infty(\Omega_0)}+||\widetilde{\varphi}_\ee||_{L^\infty(\Omega_0)}) \\
			&\leq S^2 f_3(S)g_3(S)
			\text{.}
		\end{aligned}
		\end{equation}
		Combining (\ref{eqn0}), (\ref{est1}), (\ref{est2}), (\ref{est3}), 
		(\ref{est4})
		and (\ref{eqn:hdef}),
		we obtain
		\begin{equation}\label{est2-1}
		\begin{aligned}
			&\int_{\Omega} -f_\ee (u_\ee- \widetilde{\varphi}_\ee)\, dx \\ 
			&\leq -\frac{1}{\ee} \int_{\Omega\setminus\Omega_0} 
			G'(u_\ee-\widetilde{\varphi}_\ee)(u_\ee-\widetilde{\varphi}_\ee) \, dx
			+ S(f_0(S)g_0(S) + f_2(S)g_2(S) + Sf_3(S)g_3(S)) \\
			&\leq -\frac{1}{\ee} \int_{\Omega\setminus\Omega_0} 
			G'(u_\ee-\widetilde{\varphi}_\ee)(u_\ee-\widetilde{\varphi}_\ee) \, dx + H(S) \text{.}
		\end{aligned}
		\end{equation}

		We now estimate $H(S)$ using the following inequality 
		\cite[Corollary 2.2]{LePRS}:
		\begin{equation}\label{cor22}
		\begin{aligned}
			||u_\ee||_{L^\infty(\Omega)} \leq
			C_5 + C_6 \int_{\Omega\setminus\Omega_0} |u_\ee| \, dx 
			\text{.}
		\end{aligned}
		\end{equation}
		From (\ref{eqn:sdef}) and (\ref{cor22}),
		we have
		\begin{equation}\label{est2-1-1}
		\begin{aligned}
			S \leq C_4(1 + ||u_\ee||_{L^\infty(\Omega)} )
			\leq C_7 \int_{\Omega\setminus\Omega_0} (1 + |u_\ee|) \, dx 
			\leq \frac{1}{|\Omega\setminus\Omega_0|}
			\int_{\Omega\setminus\Omega_0} C_8(1 + |u_\ee|) \, dx 
			\text{.}
		\end{aligned}
		\end{equation}
		Because $H$ is convex, 
		combining (\ref{est2-1-1}) with Jensen's inequality gives us 
		\begin{equation}\label{est2-2}
		\begin{aligned}
		    H(S)\leq H\left(\frac{1}{|\Omega\setminus\Omega_0|}
			\int_{\Omega\setminus\Omega_0} C_8(1 + |u_\ee|) \, dx \right) \\
			\leq \frac{1}{|\Omega\setminus\Omega_0|}
			\int_{\Omega\setminus\Omega_0} H(C_8 (1 + |u_\ee|)) \, dx
			\text{.}
		\end{aligned}
		\end{equation}
		Note that (\ref{eqn:Hdef}) implies 
		\begin{equation}\label{eqn:GHest}
		\begin{aligned}
			G'(u_\ee-\widetilde{\varphi}_\ee) (u_\ee-\widetilde{\varphi}_\ee)
			=2(u_\ee-\widetilde{\varphi}_\ee)^2 H((u_\ee-\widetilde{\varphi}_\ee)^2)
			\text{.}
		\end{aligned}
		\end{equation}
		As $|\widetilde{\varphi}_{\ee}|\leq C_1$,
		we have
		\begin{equation}\label{eqn:uesquare}
		\begin{aligned}
			(u_\ee-\widetilde{\varphi}_\ee)^2 \geq C_8 (1+|u_\ee|)
		\end{aligned}
		\end{equation}
		when $|u_\ee| >C_9\geq C_1 + 1$. 
		We consider the following cases:
		\begin{enumerate}[label=(\roman*)] 
			\item If $|u_\ee| > C_9$. 
				Because $H$ is increasing, 
				from (\ref{eqn:GHest}) and (\ref{eqn:uesquare}) we have 
				\begin{equation}\label{eqn:estcase1}
				\begin{aligned}
					G'(u_\ee-\widetilde{\varphi}_\ee) (u_\ee-\widetilde{\varphi}_\ee)
					\geq H((u_\ee-\widetilde{\varphi}_\ee)^2)
					\geq H(C_8(1+|u_\ee|))
					\text{.}
				\end{aligned}
				\end{equation}
			\item If $|u_\ee| \leq C_9$. 
				We have 
				\begin{equation}\label{eqn:estcase2}
				\begin{aligned}
					H(C_8(1+|u_\ee|))\leq H(C_8(1+C_9))=:C_{10}
					\text{.}
				\end{aligned}
				\end{equation}
		\end{enumerate}
		Combining (\ref{eqn:estcase1}) and (\ref{eqn:estcase2}) gives 
		\begin{equation}\label{eqn:Hcombest}
		\begin{aligned}
			H(C_8(1+|u_\ee|))
			\leq C_{10} + G'(u_\ee-\widetilde{\varphi}_\ee)(u_\ee-\widetilde{\varphi}_\ee)
			\quad\text{in }\Omega 
			\text{.}
		\end{aligned}
		\end{equation}
		Therefore, from (\ref{est2-2}) we now have, for $\ee<\ee_0$,
		where $\ee_0 = \ee_0(n,\Omega,\Omega_0,\varphi,\psi,f_k, g_k)$ 
		is small,
		\begin{equation}\label{eqn:Hsest}
		\begin{aligned}
			H(S)
			&\leq \frac{1}{|\Omega\setminus\Omega_0|}
			\int_{\Omega\setminus\Omega_0} 
			C_{10}+G'(u_\ee-\widetilde{\varphi}_\ee)(u_\ee-\widetilde{\varphi}_\ee)\, dx\\
			&\leq C_{10} + \frac{1}{2\ee}
			\int_{\Omega\setminus\Omega_0} 
			G'(u_\ee-\widetilde{\varphi}_\ee)(u_\ee-\widetilde{\varphi}_\ee)\, dx
			\text{.}
		\end{aligned}
		\end{equation}

		Now, by combining (\ref{ineq1}), (\ref{est2-1}), and (\ref{eqn:Hsest}),
		we get
		\begin{align}\label{keyest}
			\int_{\partial\Omega} \ee ((u_\ee)_\nu^+)^n \, dS
			+ \frac{1}{2\ee} \int_{\Omega\setminus\Omega_0}
			G'(u_\ee-\widetilde{\varphi}_\ee)(u_\ee-\widetilde{\varphi}_\ee) \, dx
			\leq C_2 + C_{10}
			\text{.}
		\end{align}
		We are now ready to prove the $L^\infty$ bound for $u_\ee$.
		As $G'(x) = 2xH(x^2)$ and $H(x^2) \geq x^2$,
		(\ref{keyest}) implies 
		\begin{equation}\label{ineq2fin}
		\begin{aligned} 
			C_{11}:=
			(C_{10}+C_2)\ee_0 
			&\geq \frac{1}{2} \int_{\Omega\setminus\Omega_0} G'(u_\ee - \widetilde{\varphi}_\ee)(u_\ee -\widetilde{\varphi}_\ee) \, dx \\
			&\geq \int_{\Omega\setminus\Omega_0} (u_\ee - \widetilde{\varphi}_\ee)^4 \, dx
			\text{.}
		\end{aligned}
		\end{equation}
		From (\ref{cor22}), we have
		\begin{equation*} 
		\begin{aligned} 
			||u_\ee||_{L^\infty (\Omega)} 
			&\leq C_5 + C_6 \int_{\Omega\setminus\Omega_0} |u_\ee| \, dx \\ 
			&\leq C_5+C_6|\Omega\setminus\Omega_0|\sup_{\Omega}|\widetilde{\varphi}_\ee| 
			+ C_6 \int_{\Omega\setminus\Omega_0} |u_\ee-\widetilde{\varphi}_\ee| \, dx \\ 
			&\leq C_{12}+C_{13}
			\left(\int_{\Omega\setminus\Omega_0} (u_\ee-\widetilde{\varphi}_\ee)^4 \, dx \right)^{1/4} \\
			&\leq C_{14}
			\quad\text{by  (\ref{ineq2fin}).}
		\end{aligned}
		\end{equation*}
		The proof of the lemma is complete.
	\end{proof}

	\begin{remark}\label{rmk:G}
		In Le \cite{LeCPAM,LePRS,Twisted_Harnack},
		the Lagrangian $F$ was assumed to 
		have a quadratic growth in the $\pp$ variable. 
		Especially, 
		$|F_{p_i x_i}|$ was assumed to grow linearly in $\pp$.
		Therefore, the integral in the left-hand side of (\ref{est3})
		could be bounded by a quadratic term.
		In (\ref{est3}), 
		this term cannot be bounded by a quadratic term 
		because we are assuming a general growth assumption (\ref{Fest}) for $F$. 
		This is why we had to replace the quadratic term in  (\ref{eqn:approx})
		using $G$ defined by (\ref{eqn:hdef})--(\ref{eqn:Hdef}) in (\ref{Jepsilon}).
	\end{remark}
	
	Combining the $L^\infty$ bound (\ref{eqn:linfbdd}) with 
	the gradient bound (\ref{eqn:gradestsubset}),
	we obtain the following corollary.

	\begin{corollary}
		If $x\in\Omega_0$ and $0<\ee<\ee_0$
		where $\ee_0 = \ee_0(n,\Omega,\Omega_0,\varphi,\psi,f_k, g_k)$ is small,
		then we have 
		\begin{equation}\label{eqn:gradbdcor}
		\begin{aligned}
			|Du_\ee(x)|
			\leq \frac{\sup_{\partial\Omega}|\varphi| +C_{14}}{
				\dist(\Omega_0,\partial\Omega)}
			=:C_{15}
			\text{.}
		\end{aligned}
		\end{equation}
	\end{corollary}

	From now on, we fix $\ee<\ee_0$.
	Before we move on to the next step of the proof,
	we revisit the proof of (\ref{keyest}) and note that 
	the left-hand side can be bounded by a constant 
	independent of $\ee$, 
	without having to assume that
	$u_\ee\leq\widetilde{\varphi}_\ee-1$ in $\Omega_0$. 

	\begin{remark} 
		In the proof of (\ref{keyest}),
		the inequality $u_\ee\leq\widetilde{\varphi}_\ee-1$ in $\Omega_0$ 
		was used to show (\ref{est1}).
		Having established the bounds (\ref{eqn:linfbdd}) and (\ref{eqn:gradestsubset}),
		we can obtain an estimate for the left-hand side
		of (\ref{est1}) without the assumption.

		From (\ref{Fest}), (\ref{eqn:linfbdd}) and (\ref{eqn:gradestsubset}),
		we have 
		\begin{align*} 
			0\leq (F_{p_i p_j})_{1\leq i,j\leq n}\leq f_1(C_{14})g_1(C_{15})I_n
			\text{,}
		\end{align*}
		and therefore 
		\begin{align*} 
			0\leq F_{p_i p_j} D_{ij}u_\ee \leq f_1(C_{14})g_1(C_{15})\Delta u_\ee 
			\text{.}
		\end{align*}
		By the divergence theorem, (\ref{eqn:linfbdd}) and (\ref{eqn:gradestsubset}),
		(\ref{est1}) can be replaced by 
		\begin{equation*}
		\begin{aligned}
			\int_{\Omega_0} F_{p_i p_j} D_{ij}u_\ee (u_\ee-\widetilde{\varphi}_\ee)\, dx
			&\leq \left(\norm{u_\ee}_{L^\infty(\Omega_0)}
			+\norm{\widetilde{\varphi}_\ee}_{L^\infty(\Omega_0)}\right)
			f_1(C_{14})g_1(C_{15})\int_{\Omega_0} \Delta u_\ee\, dx \\ 
			&\leq C_{16}\int_{\Omega_0}\Delta u_\ee\, dx 
			=C_{16}\int_{\partial\Omega_0} (Du_\ee\cdot{\nu_0})\, dS
			\leq C_{17}
			\text{,}
		\end{aligned}
		\end{equation*}
		where $\nu_0$ is the outer unit normal vector to $\partial\Omega_0$.
		Therefore, 
		for $u_\ee$ not necessarily satisfying $u_\ee\leq\widetilde{\varphi}_\ee-1$ in $\Omega_0$,
		we have instead of (\ref{keyest}), 
		\begin{equation}\label{keyest2}
		\begin{aligned}
			\int_{\partial\Omega} \ee ((u_\ee)_\nu^+)^n \, dS
			+ \frac{1}{2\ee} \int_{\Omega\setminus\Omega_0}
			G'(u_\ee-\widetilde{\varphi}_\ee)(u_\ee-\widetilde{\varphi}_\ee) \, dx
			\leq C_{18}
			\text{.}
		\end{aligned}
		\end{equation}
		Here $C_{18}$ is a constant possibly larger than $C_2+C_{10}$.
	\end{remark}

	Next,
	we prove the following estimates for $f_\ee$ in $\Omega_0$.
	\begin{lemma}[Estimates for $f_\ee$ in $\Omega_0$]\label{prop:fe}
		There are positive constants $\widetilde{C}$ and $D_*$
		that depend on $n, \Omega_0, \Omega, \varphi, \psi, f_k$ and $g_k$
		such that
		\begin{align} 
			-\widetilde{C}-D_*\Delta u_\ee \leq f_\ee \leq \widetilde{C} 
			\quad\text{in }\Omega_0
			\text{.}
		\end{align}
	\end{lemma}
	\begin{proof} 
		Expanding as in (\ref{fe_nondiv}),
		from (\ref{eqn:eleq}) we get 
		\begin{equation}\label{eqn:estfe_x}
		\begin{aligned}
			f_\ee &= 
			\frac{\partial F}{\partial z}(x,u_\ee, Du_\ee)\\
			&- \left(F_{p_i x_i}(x, u_\ee, Du_\ee)
				+ F_{p_i z}(x, u_\ee, Du_\ee) D_i u_\ee
				+ F_{p_i p_j} (x, u_\ee, Du_\ee) D_{ij}u_\ee \right)
				\quad\text{in }\Omega_0
			\text{.}
		\end{aligned}
		\end{equation}
		Combining (\ref{Fest}), (\ref{eqn:linfbdd}) and (\ref{eqn:gradbdcor})
		yields
		\begin{align}\label{eqn:estfe_a} 
			\left|\frac{\partial F}{\partial z}(x,u_\ee, Du_\ee)\right|\leq
			f_0(||u_\ee||_{L^\infty(\Omega_0)}) g_0(||Du_\ee||_{L^\infty(\Omega_0)}) 
			\leq f_0(C_{14})g_0(C_{15})
			\text{,}
		\end{align}
		\begin{align}\label{eqn:estfe_b}
			|F_{p_i x_i}(x,u_\ee, Du_\ee)|\leq
			f_2(||u_\ee||_{L^\infty(\Omega_0)}) g_2(||Du_\ee||_{L^\infty(\Omega_0)}) 
			\leq f_2(C_{14})g_2(C_{15})
			\text{,}
		\end{align}
		\begin{equation}\label{eqn:estfe_c}
		\begin{aligned} 
			|F_{p_i z}(x,u_\ee, Du_\ee)D_i u_\ee|
			&\leq
			f_3(||u_\ee||_{L^\infty(\Omega_0)}) 
			g_3(||Du_\ee||_{L^\infty(\Omega_0)}) 
			||Du_\ee||_{L^\infty(\Omega_0)} \\
			&\leq f_3(C_{14})g_3(C_{15}) C_{15}
			\text{,}
		\end{aligned}
		\end{equation}
		and
		\begin{align}\label{eqn:estfe_d} 
			0 \leq (F_{p_i p_j})_{1\leq i,j\leq n}  \leq
			f_1(||u_\ee||_{L^\infty(\Omega_0)}) g_1(||Du_\ee||_{L^\infty(\Omega_0)})I_n
			\leq f_1(C_{14})g_1(C_{15}) I_n =:D_* I_n
			\text{.}
		\end{align}
		From (\ref{eqn:estfe_d}), we get 
		\begin{align}\label{eqn:estfe_d2}
			0 \leq F_{p_i p_j} D_{ij}u_\ee \leq D_* \Delta u_\ee
			\text{.}
		\end{align}
		Combining (\ref{eqn:estfe_x})--(\ref{eqn:estfe_c}) with (\ref{eqn:estfe_d2})
		gives the desired inequality.
	\end{proof}

	Having established the $L^\infty$ bound in Lemma \ref{prop:linftyest}
	and the estimates for $f_\ee$ in Lemma \ref{prop:fe},
	we can carry out the rest of the proof in Le \cite[Section 2]{Twisted_Harnack}.
	We include an outline of the proof below.

	As in \cite[Lemma 2.3]{Twisted_Harnack},
	we have an upper bound for $\det D^2 u_\ee$:
	\begin{lemma}[Upper bound for $\det D^2 u_\ee$]\label{lem:detubd}
		There is $C_{19}(\ee)>0$ such that 
		\begin{align*} 
			\det D^2 u_\ee \leq C_{19}(\ee)
			\quad\text{in }\Omega\text{.}
		\end{align*}
	\end{lemma}

	From the upper bound in Lemma \ref{lem:detubd} and the boundary condition 
	$u_\ee=\varphi$ on $\partial\Omega$, 
	we can construct suitable barriers to show the gradient estimate:
	\begin{equation}
	\begin{aligned}
		|Du_\ee | \leq C_{20}(\ee)
		\quad\text{in }\Omega\text{.}
	\end{aligned}
	\end{equation}

	We need the following transformation of (\ref{eqn:eleq})
	into linearized Monge-Amp\`ere equations with drifts
	(see \cite[Lemma 2.4]{Twisted_Harnack} and \cite[Lemma 2.1]{KLWZ}):
	\begin{lemma}[Transformation of (\ref{eqn:eleq})]\label{lem:lmatransform}
		Let $x_0\in\overline{\Omega}$ be fixed. 
		Define the following functions in $\overline{\Omega}$:
		\begin{equation*} 
		\begin{aligned} 
			F_\ee^{x_0}(x) &:=
			\frac{D_* |x-Du_\ee (x_0)|^2}{2\ee}\text{,}\\
			\eta_\ee^{x_0}(x) &:=
			w_\ee(x) e^{F_\ee^{x_0}(Du_\ee(x))}\text{,}\\
			{\bf b}^{x_0}(x) &:=
			-(\det D^2u_\ee(x)) \frac{D_*}{\ee}(Du_\ee(x)-Du_\ee(x_0))
			\text{.}
		\end{aligned}
		\end{equation*}
		Then, we have
		\begin{align*} 
			U_\ee^{ij} D_{ij}\eta_\ee^{x_0} + 
			{\bf b}^{x_0} \cdot D\eta_\ee^{x_0} =
			\frac{f_\ee+D_* \Delta u_\ee}{\ee}
			e^{F_\ee^{x_0}(Du_\ee(x))}
			\quad\text{in }\Omega\text{.}
		\end{align*}
	\end{lemma}
	
	Using the transformation in Lemma \ref{lem:lmatransform}	
	in conjunction with the Aleksandrov-Bakelman-Pucci (ABP) maximum principle,
	we obtain a lower bound for $\det D^2 u_\ee$.

	\begin{lemma}[Lower bound for $\det D^2 u_\ee$]\label{lem:detlbd}
		There is $C_{21}(\ee) > 0$ such that 
		\begin{align*} 
			\det D^2 u_\ee \geq C_{21}^{-1}(\ee)
			\quad\text{in }\Omega 
			\text{.}
		\end{align*}
	\end{lemma}

	Combining the bounds on $\det D^2 u_\ee$ 
	in Lemmas \ref{lem:detubd} and \ref{lem:detlbd},
	the boundary condition $u_\ee=\varphi$ on $\partial\Omega$, 
	and the global $C^{1,\alpha}$ estimates for Monge-Amp\`ere equation
	in \cite[Proposition 2.6]{LS},
	we obtain global $C^{1,\alpha}$ estimates for $u_\ee$.

	\begin{lemma}[Global $C^{1,\alpha}$ estimates for $u_\ee$]
		There is $\alpha_0(\ee)\in (0,1)$ such that 
		$$
		\norm{u_\ee}_{C^{1,\alpha_0(\ee)}}\leq C_{22}(\ee)
		\text{.}
		$$
	\end{lemma}

	Using the transformation in Lemma \ref{lem:lmatransform}
	and the one-sided pointwise H\"older estimates at the boundary 
	for solutions to non-uniformly elliptic, 
	linear equations with pointwise H\"older continuous drift
	\cite[Proposition 2.7]{Twisted_Harnack},
	we obtain H\"older estimates for $w_\ee$ at the boundary. 
	The twisted Harnack inequality in \cite[Theorem 1.3]{Twisted_Harnack}
	gives H\"older estimates for $w_\ee$ in the interior.

	By combining the H\"older estimates for $w_\ee$ in the interior and the boundary,
	and using the boundary localization theorem of Savin \cite{S2},
	we obtain global H\"older estimates for $w_\ee$.

	\begin{lemma}[Global H\"older estimates for $w_\ee$]\label{gloholder}
	    There are constants $\alpha_1 (\ee)\in (0,1)$ and 
		$C_{23}(\ee) > 0$ so that 
		\begin{align*}
			||w_\ee ||_{C^{\alpha_1(\ee)} (\overline{\Omega})} \leq C_{23}(\ee)
			\text{.}
		\end{align*}
	\end{lemma}

	Having established the global H\"older estimates in Lemma \ref{gloholder},
	we can prove the global $W^{4,s}(\Omega)$ estimate.

	\begin{proof}[Proof of Proposition \ref{prop:unifest}]
		From  
		\begin{align*}
			\det D^2 u_\ee = w_\ee^{-1}
			\quad\text{in }\Omega
			\text{,}\quad 
			u_\ee=\varphi 
			\quad\text{on }\partial\Omega\text{,}
		\end{align*}
		the H\"older estimates in Lemma \ref{gloholder},
		and the global $C^{2,\alpha}$ estimates for the Monge-Amp\`ere equation 
		\cite{S2, TWboundaryreg},
		we get
		\begin{align*} 
			||u_\ee||_{C^{2,\alpha_1(\ee)}(\overline{\Omega})} \leq C_{24}(\ee)
			\text{.}
		\end{align*}
		Therefore, 
		$U_\ee^{ij} D_{ij}$ is an uniformly elliptic operator 
		with H\"older continuous coefficients.
		Moreover, 
		$f_\ee$ is bounded in the $L^\infty$ norm.
		Thus, from 
		$$
		U_\ee^{ij} D_{ij} w_\ee = f_\ee / \ee
		\quad\text{in }\Omega
		\text{,}\quad
		w_\ee = \psi 
		\quad\text{on } \partial\Omega
		\text{,}
		$$ 
		we obtain estimates for $w_\ee$ in $W^{2,s}(\Omega)$.
		The $W^{4,s}(\Omega)$ estimate in (\ref{eqn:aprioriest}) follows.
	\end{proof}

	We are now ready to prove Theorem \ref{mainthm}\ref{thm:solvability}. 
	\begin{proof}[Proof of Theorem {\ref{mainthm}\ref{thm:solvability}}] 
		From the a priori estimate (\ref{eqn:aprioriest}) in Proposition \ref{prop:unifest},
		we can use Leray-Schauder degree theory as in Le \cite[pp.2275--2276]{LeCPAM}
		to prove the existence of a uniformly convex $W^{4,s}(\Omega)$ solution 
		$u_\ee$ to (\ref{eqn:eleq}) for all $s\in (n,\infty)$.
	\end{proof}

\section{Convergence of solutions to a minimizer}\label{sec:approx}

In this section,
we prove Theorem \ref{mainthm}\ref{thm:approx} on the convergence of 
solutions of (\ref{eqn:eleq}) to a minimizer of the variational problem 
(\ref{eqn:hdvar})--(\ref{eqn:hdvarcon}). 
We will follow the proof in Le \cite{LePRS,LeCPAM}.  

\begin{proof}[Proof of Theorem \ref{mainthm}\ref{thm:approx}]

By (\ref{eqn:linfbdd}),
the family $(u_\ee)_{\ee>0}$ of $W^{4,s}(\Omega)$ solutions to (\ref{eqn:eleq})
satisfies, whenever $0<\ee<\ee_0$,
\begin{equation}\label{eqn:linfty-subseq}
\begin{aligned}
	\norm{u_\ee}_{L^\infty(\Omega)}\leq C
\end{aligned}
\end{equation}
for $C$ independent of $\ee$.  
Furthermore, for any $\Omega '\Subset \Omega$,
we can combine (\ref{eqn:linfty-subseq}) with the gradient bound (\ref{gradbound})
to obtain 
\begin{equation}\label{eqn:gradbd-subseq}
\begin{aligned}
	\norm{Du_\ee}_{L^\infty(\Omega ')}
	\leq \widehat{C} (\Omega', \Omega)
	\text{.}
\end{aligned}
\end{equation}

From (\ref{eqn:linfty-subseq}) and (\ref{eqn:gradbd-subseq}), 
by passing to a subsequence $\ee_k\rightarrow 0$, 
we have 
\begin{equation}\label{subseq}
\begin{aligned}
	u_{\ee_k} \rightarrow u  
	\quad &\text{weakly in } W^{1,2}(\Omega_0)
	\text{,}\quad\text{and }\\
	u_{\ee_k} \rightarrow u  
	\quad &\text{uniformly on compact subsets of } \Omega
	\text{,}
\end{aligned}
\end{equation}
for some convex function $u$ in $\Omega$.
Combining (\ref{eqn:Hdef}) with (\ref{keyest2}) yields
\begin{equation*}
\begin{aligned}
	C_{18}\ee_k
	&\geq\frac{1}{2}
	\int_{\Omega\setminus\Omega_0} G'(u_{\ee_k}-\widetilde{\varphi}_{\ee_k})(u_{\ee_k}-\widetilde{\varphi}_{\ee_k}) \, dx
	=\int_{\Omega\setminus\Omega_0} H((u_{\ee_k}-\widetilde{\varphi}_{\ee_k})^2)(u_{\ee_k}-\widetilde{\varphi}_{\ee_k})^2 \, dx   \\
	&\geq\int_{\Omega\setminus\Omega_0} (u_{\ee_k}-\widetilde{\varphi}_{\ee_k})^4 \, dx
	\text{.}
\end{aligned}
\end{equation*}
Therefore, 
$\int_{\Omega\setminus\Omega_0} (u_{\ee_k}-\widetilde{\varphi}_{\ee_k})^4 \, dx\rightarrow 0$ 
as $k\rightarrow\infty$. 
Because $u_{\ee_k}\rightarrow u$ uniformly on compact subsets of $\Omega$ and 
\begin{align*} 
	\widetilde{\varphi}_{\ee_k}=\varphi+{\ee_k}^{\frac{1}{3n^2}} (e^\rho -1)
	\rightarrow \varphi
	\quad\text{as }k\rightarrow\infty
\end{align*}
uniformly in $\Omega$,
we have $u=\varphi$ in $\Omega\setminus\Omega_0$ and hence $u\in\convset$.
Now, we show that $u$ minimizes $J$ in (\ref{eqn:hdvar}) 
among the competitors in $\convset$ by the following steps.

\emph{Step 1.} First, we show that
\begin{align}\label{ineq3-1}
	\liminf_{k\rightarrow\infty} J(u_{\ee_k}) \geq J(u)
	\text{.}
\end{align}
From the convexity of $F$ in $z$ and $\pp$, 
we have 
\begin{equation}\label{eqn:jconvineq} 
\begin{aligned}
	J(u_{\ee_k})-J(u)
	&=\int_{\Omega_0} F(x,u_{\ee_k},Du_{\ee_k})-F(x,u,Du)\, dx \\
	&\geq\int_{\Omega_0} F_z(x,u,Du)(u_{\ee_k}-u)+F_{\pp}(x,u,Du)\cdot (Du_{\ee_k}-Du)\, dx
	\text{.}
\end{aligned}
\end{equation}
Here $F_\pp = (F_{p_1},\cdots,F_{p_n})$.
Therefore, from (\ref{subseq}) the right-hand side of (\ref{eqn:jconvineq})
converges to $0$ as $k\to\infty$, 
and the desired inequality (\ref{ineq3-1}) follows.

\emph{Step 2.} Next, we show that if $v$ is a convex function in $\overline{\Omega}$ 
satisfying $v=\varphi$ in a neighborhood of $\partial\Omega$, then 
\begin{equation}\label{eqn:jepsilonest1}
\begin{aligned}
	J_\ee (v) - J_\ee (u_\ee)
	\geq \ee\int_{\partial\Omega} \psi U_\ee^{\nu\nu} \partial_\nu
		(u_\ee-\varphi)\, dS 
	+\int_{\partial\Omega_0} (v-u_\ee)F_{\pp}(x,u_\ee,Du_\ee)\cdot\nu_0\, dS 
	\text{,}
\end{aligned}
\end{equation}
where $U_\ee^{\nu\nu}=U_\ee^{ij}\nu_i \nu_j$. 
Here $\nu$ and $\nu_0$ are outer unit normal vectors to 
$\partial\Omega$ and $\partial\Omega_0$.

We use mollification as in Le \cite[p.2277]{LeCPAM}
to obtain a sequence of uniformly convex $C^3(\overline{\Omega})$ functions 
$(v_h)_{h>0}$ that satisfy, for all $k\leq 2$,
\begin{equation}\label{eqn:vapprox}
\begin{aligned}
	D^k v_h \rightarrow D^k v
	\quad\text{as }h\rightarrow 0
	\quad\text{in a neighborhood of }\partial\Omega 
	\text{.}
\end{aligned}
\end{equation}
Recall from (\ref{Jepsilon}) that 
\begin{align*} 
	J_\ee (v) = \int_{\Omega_0} F(x,v(x), Dv(x)) \, dx
	+ \frac{1}{\ee} \int_\omout G(v - \widetilde{\varphi}_\ee) \, dx 
	- \ee \int_\Omega \log \det D^2 v(x) \, dx
	\text{.}
\end{align*}
First, 
by \cite[(5.9)]{LeCPAM} we have
\begin{equation}\label{bdryint1}
\begin{aligned}
	&-\int_\Omega\log\det D^2v_h\, dx 
	+\int_\Omega\log\det D^2u_\ee\, dx  \\
	&\geq\int_\Omega \ee^{-1}f_\ee (u_\ee-v_h) \, dx
	-\int_{\partial\Omega} D_i w_\ee U^{ij}_\ee (u_\ee-v_h)\nu_j\, dS
	+\int_{\partial\Omega} \psi U_\ee^{\nu\nu} \partial_\nu (u_\ee-v_h)\, dS
\end{aligned}
\end{equation}
Furthermore, using the convexity of $F$ and 
integrating by parts, we get
\begin{equation}\label{bdryint3}
\begin{aligned}
	&\int_{\Omega_0} F(x,v_h,Dv_h)
	- \int_{\Omega_0} F(x,u_\ee, Du_\ee) \\
	& \geq \int_{\Omega_0} F_z(x,u_\ee,Du_\ee)(v_h-u_\ee) \, dx
	+\int_{\Omega_0}F_\pp(x,u_\ee,Du_\ee) \cdot D(v_h-u_\ee) \, dx \\
	& =\int_{\Omega_0} F_z(x,u_\ee,Du_\ee)(v_h-u_\ee) \, dx
	-\int_{\Omega_0} \frac{\partial}{\partial x_i}
	\left(F_{p_i}(x,u_\ee,Du_\ee)\right)(v_h-u_\ee)\, dx \\
	& +\int_{\partial\Omega_0} (F_\pp(x,u_\ee,Du_\ee)\cdot \nu_0)(v_h-u_\ee)\, dS
	\text{.}
\end{aligned}
\end{equation}
Finally, from the convexity of $G$, we can conclude 
\begin{align}\label{bdryint2}
	\frac{1}{\ee}\int_{\Omega\setminus\Omega_0}G(v_h-\widetilde{\varphi}_\ee)\,dx-
	\frac{1}{\ee}\int_{\Omega\setminus\Omega_0}G(u_\ee-\widetilde{\varphi}_\ee)\,dx
	\geq \frac{1}{\ee} \int_{\Omega\setminus\Omega_0}
	G'(u_\ee-\widetilde{\varphi}_{\ee_k}) (v_h-u_\ee)\, dx
	\text{.}
\end{align}

Combining (\ref{bdryint1})--(\ref{bdryint2}),
we get after cancellation,
\begin{equation}\label{eqn:jepsilonest2}
\begin{aligned}
	J_\ee(v_h)-J_\ee(u_\ee)
	&\geq-\ee\int_{\partial\Omega}
	D_iw_\ee U_\ee^{ij} (u_\ee-v_h)\nu_j\, dS  \\
	&+\ee\int_{\partial\Omega}\psi U_\ee^{\nu\nu}
	\partial_\nu (u_\ee-v_h) \, dS 
	+\int_{\partial\Omega_0} 
	(v_h-u_\ee)F_\pp(x,u_\ee,Du_\ee)\cdot\nu_0 \, dS
	\text{.}
\end{aligned}
\end{equation}
By (\ref{eqn:vapprox}),
the right-hand side of (\ref{eqn:jepsilonest2})
converges to the right-hand side of (\ref{eqn:jepsilonest1})
as $h\rightarrow 0$. 
Furthermore, 
we have from \cite[(5.8)]{LeCPAM},
\begin{align*} 
	J_\ee (v_h)\rightarrow J_\ee (v)
	\quad\text{as }h\rightarrow 0 
	\text{.}
\end{align*}
Therefore, letting $h\to 0$ in (\ref{eqn:jepsilonest2})
gives (\ref{eqn:jepsilonest1}).

\emph{Step 3.}
We show that 
for any $v\in\overline{S}[\varphi,\Omega_0]$,
\begin{equation}\label{eqn:jepsilonest3}
\begin{aligned}
	J(v)\geq \liminf_{k\to\infty}J(u_{\ee_k})
	\text{.}
\end{aligned}
\end{equation}

We use (\ref{eqn:jepsilonest1})
to argue as in Le \cite[pp.372--374]{LePRS}
to obtain the following inequality:
\begin{equation}\label{eqn:jepsilonest4}
\begin{aligned}
	J(v) \geq \liminf_{k\rightarrow\infty}J(u_{\ee_k})
	-\limsup_{k\rightarrow\infty}
	\left[ {\ee_k}^{(n-1)/n} \eta_{\ee_k} + {\ee_k}^{1/n}\eta_{\ee_k}^{n-1}\right]
	\text{,}
\end{aligned}
\end{equation}
where 
\begin{align*}
	\eta_\ee := \ee^{1/n} \left(\int_{\partial\Omega} ((u_\ee^+)_\nu)^n \, dS\right)^{1/n}
	\text{.}
\end{align*}
From (\ref{keyest2}),
$\eta_\ee$ is bounded independent of $\ee$: 
$$
\eta_\ee \leq C_{18}^{1/n}
\text{,}
$$
and therefore (\ref{eqn:jepsilonest3}) follows from (\ref{eqn:jepsilonest4}).

\emph{Step 4.} We now show the minimality of $u$. 
Combining (\ref{ineq3-1}) and (\ref{eqn:jepsilonest3}) gives 
\begin{align*} 
	J(v)\geq \liminf_{k\rightarrow\infty} J(u_{\ee_k})
	\geq J(u)
\end{align*}
for all $v\in\convset$, 
which proves the minimality of $u$.
The proof of the theorem is complete. 

\end{proof}

\begin{remark}\label{rmk:end}
	By the technique in \cite{Kim1},
	Theorem \ref{mainthm} is also true for $n=1$.
	However, as mentioned in \cite[Remark 3.2]{Kim1},
	the proof of Theorem \ref{mainthm}\ref{thm:approx}
	requires an additional argument different from the 
	ones used above.
\end{remark}

\subsection*{Acknowledgments}
The author wishes to thank Professor Nam Q. Le
for suggesting the problem
and for the insightful guidance and support received during the work.
The author also expresses his gratitude to the anonymous referee
for carefully reading the note and providing constructive feedback.

The research of the author was supported in part by NSF grant DMS-2054686.

\end{document}